\documentclass[12pt]{amsart}
\usepackage{fullpage}
\usepackage{multirow, booktabs}

\usepackage{lipsum}
\usepackage[utf8]{inputenc}

\usepackage{amsfonts}
\usepackage[T1]{fontenc}
\usepackage{graphicx}
\usepackage{xcolor}

\usepackage{array}
\usepackage{booktabs}
\usepackage{multirow}
\usepackage{enumitem}

\usepackage[font={small,it}]{caption}

\usepackage{comment}

\newtheorem{theorem}{Theorem}
\newtheorem{proposition}[theorem]{Proposition}
\newtheorem{claim}[theorem]{Claim}
\newtheorem{property}[theorem]{Property}
\newtheorem{corollary}[theorem]{Corollary}
\newtheorem{conjecture}[theorem]{Conjecture}

\newtheorem{lemma}[theorem]{Lemma}

\theoremstyle{definition}

\def\wtW{\widehat{W}}

\def\cC{\mathcal{C}}

\def\cU{\mathcal{U}}

\usepackage{calrsfs}
\DeclareMathAlphabet{\pazocal}{OMS}{zplm}{m}{n}

\usepackage{graphicx}

\usepackage{tikz}
\usetikzlibrary{positioning}
\usetikzlibrary{decorations.pathreplacing}

\tikzset{partition/.style={fill,circle,inner sep=1pt},
         part/.style={baseline=0,scale=0.5,bend left=45},
         partlabel/.style={below}}
\usetikzlibrary{shapes,arrows}
\tikzstyle{pnt}=[draw,ellipse,fill,inner sep=1pt]

\tikzstyle{opnt}=[ ]
\tikzstyle{pntt}=[draw,ellipse,fill,inner sep=0.5pt]
\tikzstyle{point}=[draw,ellipse,fill,inner sep=2pt]

\title[Bijections for Weyl Chamber walks]{Bijections for Weyl Chamber walks ending on an axis, using arc diagrams and Schnyder woods}

\author{Julien Courtiel, Eric Fusy, Mathias Lepoutre, and Marni Mishna}

\address{JC: LIPN, Universit\'e Paris 13, France; 
EF, ML: LIX, \'Ecole Polytechnique, France;
MM: Department of Mathematics, Simon Fraser University, Canada}

\begin{document}

\begin{abstract}{ 
In the study of lattice walks there are several examples of enumerative equi\-valences
which amount to a trade-off between domain and endpoint constraints. We present
a family of such bijections for simple walks in Weyl chambers which
use arc diagrams in a natural way. One consequence is a set of new bijections
for standard Young tableaux of bounded height. A modification of the
argument in two dimensions yields a bijection between Baxter
permutations and walks ending on an axis, answering a recent question
of Burrill et al. (2016).
Some of our arguments (and related results) are proved
using Schnyder woods.
Our strategy for simple walks extends to any dimension and
    yields a new bijective connection between standard Young tableaux
    of height at most $2k$ and certain walks with prescribed endpoints
    in the $k$-dimensional Weyl chamber of type D.  
}

\end{abstract}
\maketitle

\section{Introduction}
\label{sec:introduction}
In the context of directed 2D lattice paths with unit steps, there is
a classic bijection between \emph{meanders\/} and \emph{bridges\/} of
equal length. This maps lattice walks with steps~$(1,1)$ and~$(1,-1)$
starting at the origin, staying above the $x$-axis (meanders) to those
ending at height zero (bridges) -- see Figure~\ref{fig:BasicBijection}.
\begin{figure}[ht]
\center
\includegraphics[width=\textwidth]{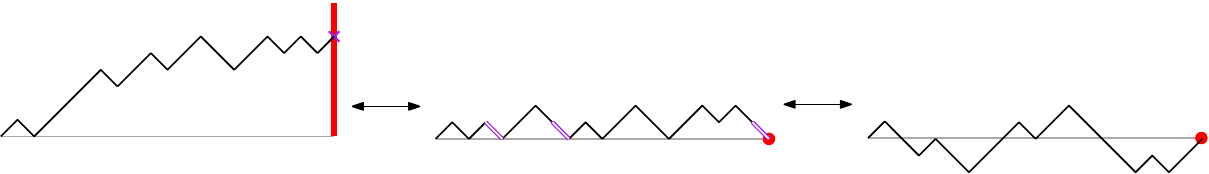}
\caption{An example of the classical bijection between meanders and bridges.
From a Dyck walk $D$ with some marked steps $d_1,\ldots,d_k$ reaching 
the $x$-axis, one gets (bijectively) a meander by 
turning every marked step into an up-step. 
To get (bijectively) a bridge from $D$, for $1\leq i\leq k$ 
we let $u_i$ be the up-step matched
with $d_i$, and we switch every step between $u_i$ and $d_i$ (included).}
 \label{fig:BasicBijection}
\end{figure}
This example illustrates a common trade-off in lattice walks between
domain constraints and endpoint constraints~\cite{Eliz15,
  BoMi10}
.  Note that the natural bijection shown in
Figure~\ref{fig:BasicBijection} proceeds via an intermediate class of
walks (Dyck walks) where both the stronger domain and endpoint restrictions are
imposed, and the elements of this class carry additional
``decorations'' (here, marked down-steps reaching the $x$-axis).

This paper introduces a similar but new strategy which can be successfully 
applied to several models of Weyl chamber walks. In particular, for two classical step sets (simple walks
and hesitating walks), we have found explicit bijections that 
exchange a domain constraint with an endpoint constraint.  In the two-dimensional case, these bijections match walks in the quadrant $\{x\geq 0,\ y\geq 0\}$ ending at the
origin (\emph{excursions}\footnote{More generally, we use the term excursion
  to indicate the set of walks with a prescribed start and end
  point. When they are not specified, the prescribed start and end is
  assumed to be the origin.}), and walks in the octant $\{x\geq y\geq 0\}$ and ending on
the $x$-axis (\emph{axis-walks}). For both step sets, these bijections pass through decorated excursions restricted to the octant. 
Deciding exactly how to mark the steps in the decorated intermediary
is less obvious than the Dyck walk example. We do this by using
\emph{open arc diagrams\/} (corresponding to partial
matchings and set partitions)
that are associated to the walks via the robust bijection of Chen
\emph{et al.}~\cite{ChDeDuStYa07} --- or more precisely, via its extension to open arc
diagrams due to Burrill~\emph{et al.\/}~\cite{BuCoFuMeMi15}. In their
full generality, these bijections map open arc diagrams with no
$(k+1)$-crossing\footnote{A $k$-crossing is a set of~$k$ mutually
  crossing arcs.} to walks in the $k$-dimensional Weyl chamber of type C $\{(x_1, \dots, x_k): x_1\geq \cdots\geq x_k\geq 0\}$ that end on the
$x_1$-axis, where the number of open arcs gives the abscissa of the
endpoint.  It is at the level of arc diagrams that the marking of the
object is easiest to describe: we map walks that end on the $x$-axis
to open arc diagrams, mark the location of the open arcs, remove them,
and then apply the inverse bijection to get marked excursions. The
schematic outline of our core idea is illustrated in
Figure~\ref{fig:remOpArc}.
The advantage of our approach is that it
very easily generalizes to walks in arbitrary dimension.

\begin{figure}
\centering
 \includegraphics[width=\textwidth]{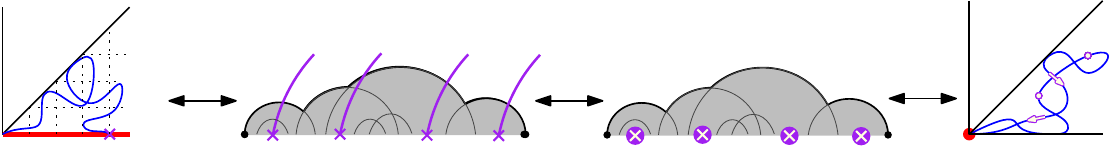}
 \caption{In the first part of our bijections for 2D walks, we map
   axis-walks in the octant to open arc diagrams with no $3$-crossing;
   we mark the positions of the open arcs, and then remove them; then we
   apply the inverse bijection on the resulting arc diagram with no
   open arcs, which yields an excursion with some markings.}
 \label{fig:remOpArc}
\end{figure}

Once these decorated excursions are obtained, it remains to process the marks.  
This processing is handled differently for
simple walks and for hesitating walks (where a further
step of transfer of decorations is needed), but in both
cases the marks are used to produce an unmarked walk in a larger
domain. 

Part of the bijection for simple walks has the nice feature
that it extends to higher dimension, unveiling a new bijective
connection with standard Young tableaux of even-bounded height
(which are known to be related 
to Weyl chamber axis-walks~\cite{Gouy86,BuCoFuMeMi15,Krat16}).




\subsection{Bijection for 2D simple walks}
A lattice model is said to be \emph{simple\/} if the step set consists
of all of the elementary vectors and their negatives. In two dimensions,
the steps correspond to the compass directions, that we denote $N, E,
S, W$.  Our first main result is the following Theorem, which is proved
in Section~\ref{sbs:mainOsc}.
\begin{theorem}
\label{thm:MainSimple}
There exists an explicit bijection (preserving the length) between
simple axis-walks of even length staying in the first octant, and
simple excursions staying in the first quadrant.
\end{theorem}

As announced in the introduction, our strategy uses open arc diagrams
to turn the simple axis-walk of length $2n$ into a decorated
excursion. This is then transformed to a simple walk of length $2n$ in
the tilted quadrant $\{(x,y): x\geq 0, \ |y|\leq x\}$ starting and
ending at $(1/2,1/2)$, and finally mapped to a pair of Dyck paths of
respective lengths $2n+2$ and $2n$.  These are
known~\cite{CoDuVi86,Bern07} to be in bijection with simple excursions of length $2n$ in the quadrant. This chain of bijections is illustrated by Figure~\ref{fig:th1}.

\begin{figure}
 \includegraphics[width=0.95\textwidth]{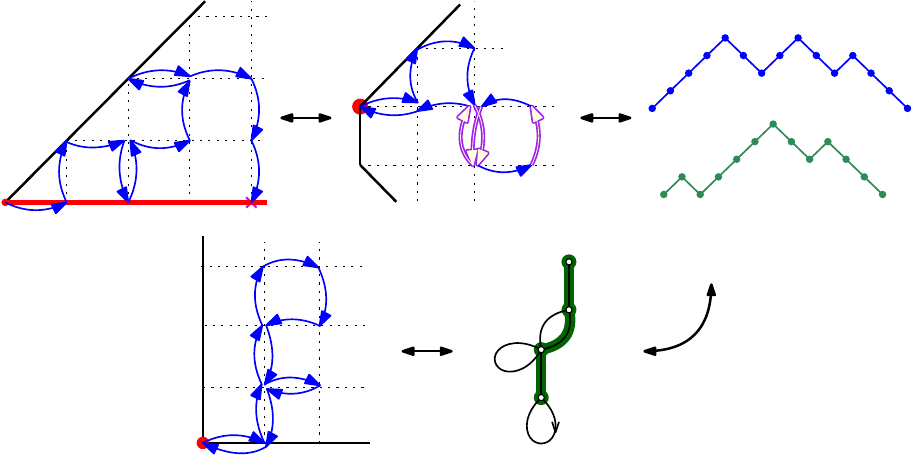}
 \caption{Illustration of the bijective proof of
   Theorem~\ref{thm:MainSimple}.  The top part connects simple axis-walks in the octant and pairs of Dyck paths, as explained by Theorem~\ref{theo:gouyou}.  The rest comes from the works of Bernardi~\cite{Bern07}, where he links pair of Dycks paths to simple excursions in the quadrant via the so-called \textit{tree-rooted planar maps}. Note that this bijection seems to lose track of the statistic \textit{"x-coordinate of the endpoint"}. }
 \label{fig:th1}
\end{figure}

Compare Theorem ~\ref{thm:MainSimple} to the following
result recently proved by Elizalde:
\begin{theorem}[Elizalde~\cite{Eliz15}]\label{thm:eliz}
  There exists an explicit bijection (preserving the length) between
  simple walks staying in the first octant and ending on the diagonal,
  and simple excursions staying in the first quadrant.
\end{theorem}
In Section~\ref{sec:SchWd} we provide an alternative proof of
Theorem~\ref{thm:eliz} using Schnyder woods.  Note that
Theorems~\ref{thm:MainSimple} and~\ref{thm:eliz} together yield a
bijection for simple walks of length $2n$ staying in the octant,
mapping those ending on the $x$-axis to those ending on the diagonal.
This answers an open question of Bousquet-M\'elou and
Mishna~\cite{BoMi10}.

Moreover, in Section~\ref{sec:young} we give an extension for
dimension $k\geq 1$ of the aforementioned bijection between simple
axis-walks in the octant and simple walks from
$(\frac 1 2, \frac 1 2)$ to itself in the tilted
quadrant. 
This yields a new bijective connection between standard Young tableaux
of height at most $2k$ and simple excursions in the $k$-dimensional
Weyl chamber of type D.

Grabiner and Magyar gave explicit enumeration formulas for excursions
in Weyl Chambers, and hence this bijection permits a straightforward
application of their results. In Section~\ref{sec:Gessel} we use their results
to illustrate a new derivation of Gessel's formulas for standard Young
tableaux of even-bounded height.

\subsection{Bijection for 2D hesitating walks.}
A (2-dimensional) \emph{hesitating} walk is a sequence of steps
$s_1,\ldots,s_{2n}$ such that every step of odd index is either in
$\{N,E\}$ (a \emph{positive} step) or is $\mathbf{0} =(0,0)$, every step of
even index is either in $\{W,S\}$ (a \emph{negative step}) or is~$\mathbf{0}$, and
for every $i\in\{1,\ldots,n\}$, $s_{2i-1}$ and $s_{2i}$ cannot both be
zero. It is convenient to not represent the null step in the drawings,
but to group the steps by pairs of the form
$(s_{2i-1},s_{2i})$ instead. In Section~\ref{sbs:mainHes} we show the
analogous, although more difficult, result for hesitating walks, which
answers a recent question of Burrill \emph{et al.}~\cite{BuCoFuMeMi15}: 

\begin{theorem}
\label{thm:MainBaxter}
There exists an explicit bijection (preserving the length) between
hesitating axis-walks in the first octant, and hesitating excursions
in the first quadrant.
\end{theorem}
It is then easy to derive   
 a bijection between Baxter permutations of size $n+1$ (known to be
in bijection with hesitating excursions of half-length $n$ 
in the quadrant)    
and open matching-diagrams with $n$ points and no enhanced 3-nesting 
(known to be in bijection with 
hesitating axis-walks of half-length $n$ in the  octant).  This 
answers a conjecture formulated by Burrill \emph{et al.}~\cite{BuElMiYe15}. 

In order to show Theorem~\ref{thm:MainBaxter},  
again the first step is to use the strategterny of
Figure~\ref{fig:remOpArc} to turn the axis-walks into decorated
hesitating excursions, where the decoration consists in marking some
W-steps on the $x$-axis. A further ingredient here is to turn the
decoration into marked steps leaving the diagonal, after which the
decorated excursions in the octant are known~\cite{BuCoFuMeMi15} to be
equivalent to hesitating excursions in the quadrant.
 
Hesitating excursions of half-length $n-1$ in the quadrant are known
to be counted by the Baxter numbers
$B_n=\frac{2}{n(n+1)^2}\sum_{k=1}^n\binom{n+1}{k+1}\binom{n+1}{k}\binom{n+1}{k-1}$.
Indeed, as shown by Burrill \emph{et al.}~\cite{BuCoFuMeMi15}, they are in easy bijection
with the classical Baxter family of non-intersecting triples of
directed lattice walks.
On the other hand it has been first shown in~\cite{XiZh09} (and more
recently in~\cite{BuCoFuMeMi15}) that hesitating axis-walks of
half-length $n$ in the octant are also counted by $B_{n+1}$.  Both of
these proofs involve an equality of generating functions, and
neither proof retains significant combinatorial intuition. Our result
is the first bijective proof that these walks are counted by
$B_{n+1}$.  Such a result is not obvious to find since the family of
hesitating axis-walks in the octant does not seem to be naturally
endowed with the classical (bivariate) symmetric generating tree
common to the Baxter families such as Baxter permutations, twin pairs
of binary trees, 2-oriented plane quadrangulations, and plane bipolar
orientations~\cite{dulucq1998baxter, AcBaPi06, FeFuNoOr11}. These
families share the same generating tree, and hence there exists a
``canonical'' bijection relating them ~\cite{BoBoFu09}. We cannot rely
on such a systematic bijective strategy here.

Theorem~\ref{thm:MainBaxter} can be extended to a similar kind of plane walks, namely vacillating walks, leading to some new enumerative results on such walks.

In the case of simple walks, we saw that both the excursions in the quadrant and the axis-walks in the octant are in bijection with the simple walks in the octant ending on the diagonal. Does this hold for hesitating walks? In fact, a computational enumeration (up to half-length $50$) suggests the following conjecture (where the 
conjectural part is the bijective link to the third family).

\begin{conjecture}\label{conj:diagonal_hesitating}
 The following families are in bijection: 
 \begin{itemize}
  \item hesitating excursions of length $2n$ in the quadrant,
  \item hesitating axis-walks of length $2n$ in the octant, 
  \item hesitating walks of length $2n$ in the octant, 
  ending on the \emph{thick diagonal} : 
$\{(n,n),n\in\mathbb{N}\}\bigcup\{(n+1,n),n\in\mathbb{N}\}$.
 \end{itemize}
\end{conjecture}
If we denote by $u_n$ (resp. $v_n$) the  number of hesitating walks of length $2n$ in the octant ending on $\{x=y\}$ 
(resp. ending on $\{x=y+1\}$), then Conjecture~\ref{conj:diagonal_hesitating} would imply that $u_n+v_n$ equals
the Baxter number $B_{n+1}$. We have not been able yet to find a 
computational proof that $u_n+v_n=B_{n+1}$, but have numerically
checked that both $u_n$ and $v_n$ seem to be P-recursive 
(of order $2$). We have not found any natural subfamily of hesitating
axis-walks of length $2n$ (for instance with a parity
constraint on the ending point) that are counted by~$u_n$.  



\section{Open arc diagrams}
\label{sec:def}
Arc diagrams are a graphic representation of combinatorial structures
such as partitions or matchings, which enables a convenient visualization
of certain patterns, such as crossings. A \textit{partition diagram}
is defined for a set partition $\pi$ of $\{1,\dots,n\}$: draw $n$
points on a line, labeled from $1$ to $n$; for each (ordered) block
$\{a_1,\dots,a_k\}$ of $\pi$, we draw an arc from $a_i$ to $a_{i+1}$
for $1\leq i\leq k-1$. A \textit{matching diagram} is a partition
diagram where the underlying set partition is a matching (i.e. every
block has size $2$).

A point of a partition diagram can be an \emph{opening point}, if it
is the first point of a block of size $\geq 2$; a \emph{closing
  point}, if it is the last point of a block of size $\geq 2$; a
\emph{transition point}, if it is a non-extremal point in a block of
size $\geq 3$; or a \emph{fixed point}, if it is a block of size
$1$. A matching diagram only has opening and closing points.

\begin{figure}\center
\includegraphics[width=0.9\textwidth]{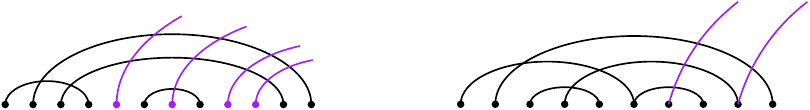}
\caption{Left: an open matching-diagram with $4$ open arcs.
Right: an open partition-diagram with $2$ open arcs.}
\label{fig:exArcDgm}
\end{figure}

A \emph{$3$-crossing} pattern in an arc diagram is a set of three
mutually crossing arcs, i.e. three arcs $(i_1,j_1)$, $(i_2,j_2)$ and
$(i_3,j_3)$ with $i_1 < i_2 < i_3 < j_1 < j_2 < j_3$. An
\textit{enhanced $3$-crossing} is a $3$-crossing where arcs sharing an
endpoint are also considered to be crossing. More formally, three arcs
$(i_1,j_1)$, $(i_2,j_2)$, $(i_3,j_3)$ form an enhanced $3$-crossing if
$i_1 < i_2 < i_3 \leq j_1 < j_2 < j_3$. These definitions are
naturally generalized to $k$-crossings for $k \geq 2$.

Arc diagrams can be extended to \textit{open arc diagrams}, by
allowing arcs with only a left endpoint, and no right endpoint (see
Figures~\ref{fig:exArcDgm} for examples). A $3$-crossing and an enhanced $3$-crossing can now include one open arc but only if this open arc starts as position $i_3$ (informally it is as if $j_3=+\infty$).
The two examples of Figure~\ref{fig:exArcDgm} have multiple $2$-crossings 
but no $3$-crossing.
For all types of arc diagrams the \emph{size} is defined as the number
of points in the diagram.

In~\cite{ChDeDuStYa07}, Chen \textit{et al.} describe a bijection
between arc diagrams with no $(k\!+\!1)$-crossings, and excursions staying
in the $k$-dimensional Weyl chamber of type $C$. 
It was subsequently extended in~\cite{Burr14,BuCoFuMeMi15} by
Burrill~\textit{et al.} to map open arc diagrams to axis-walks (see the left part of Figure \ref{fig:remOpArcHes} for an example).

\begin{theorem}[Burrill \emph{et al.}~\cite{BuCoFuMeMi15}, restricted
  to $3$-crossings]
 \label{thm:BuCoFuMeMi}
 There exists an explicit bijection between open
 matching (resp. open partition) diagrams of size $n$, with $m$ open
 arcs and no $3$-crossing (resp. no enhanced $3$-crossing), and simple
 (resp. hesitating) walks of length~$n$ (resp. of half-length $n$)
 staying in the first octant $\{(x,y), 0 \leq y \leq x \}$, starting
 at the origin and ending at $(m,0)$.
\end{theorem}

We refer the reader to the original paper~\cite{BuCoFuMeMi15} for a description of the
bijection in its complete form,  
which is based on the Robinson-Schensted insertion
algorithm on tableaux, but can also conveniently be reformulated in
terms of growth diagrams~\cite{Krat06,Krat16}.

Here are important properties of this correspondence that we use in
our bijections. The first part is extracted from Proposition 3 of~\cite{BuCoFuMeMi15}, and the second part follows straightforwardly from the description of the bijection.
\begin{property}
 \label{ppt:cddsy}
 Let $\pi$ be a closed matching (resp. partition) diagram of size~$n$
 with no $3$-crossing (resp. enhanced $3$-crossing), and $\omega$ the
 simple (resp. hesitating) walk of length $n$ corresponding to $\pi$
 via the bijection  due to Burrill \emph{et al.} ~\cite{BuCoFuMeMi15}.
 
 
 Open arcs can be inserted into intervals at positions
 $i_1,i_2,\dots,i_k$ in~$\pi$ without forming a $3$-crossing if and
 only if for every~$j \in \{1,\dots,k\}$ the~$y$-coordinate
 after~$i_j$ steps in~$\omega$ is zero.

 Given a partition diagram~$\pi$, the fixed points of~$\pi$ correspond
 to the factors~$EW$ in~$\omega$, and the closing points of $\pi$
 correspond to the factors $\{\mathbf{0} W, \mathbf{0} S\}$ in
 $\omega$ (with $\mathbf{0}$ denoting the zero step).  In addition, an
 open arc can be added on a fixed point or a closing point without
 creating an enhanced $3$-crossing if and only if an open arc could be
 added into the interval just to the left of that point without creating a $3$-crossing.
 \end{property}

\section{Proof of Theorem~\ref{thm:MainSimple}: Simple Walks}
\label{sbs:mainOsc}
The first main ingredient lies in the results
from~\cite{CoDuVi86,Bern07}, where the respective authors describe a
correspondence between simple excursions of length $2n$ in the
quadrant, and pairs of Dyck paths of lengths $2n$ and $2n+2$. To have
a bijective proof of Theorem~\ref{thm:MainSimple}, we then need to
connect such pairs of Dyck paths to simple axis-walks of even length in the octant.  This is given by the following theorem, with an
extension to odd length.

\begin{theorem}\label{theo:gouyou}
  Let $\cC_n$ be the set of Dyck paths of length $2n$, and let $\cU_n$
  be the set of simple axis-walks of length $n$ in the first octant.
   There is an explicit bijection for each
  $n\geq 0$ between $\cU_{2n}$ and $\cC_n\times\cC_{n+1}$, and between
  $\cU_{2n+1}$ and $\cC_{n+1}\times\cC_{n+1}$.
\end{theorem}
This section provides a bijective proof of this result, as illustrated
by Figure~\ref{fig:exOsc}. Gouyou-Beauchamps~\cite{Gouy86} showed that
the cardinality of simple axis-walks in the octant is indeed
$\mathrm{Cat}_n\mathrm{Cat}_{n+1}$ or $\mathrm{Cat}_{n+1}\ \!\!^2$,
depending on the parity, where $\mathrm{Cat}_n$ is the $n$th Catalan
number. However, his proof uses a reflection principle argument of
Gessel Viennot which
involves subtractions and cancellations of terms.


\begin{figure}\center
 \includegraphics[width=.95\textwidth]{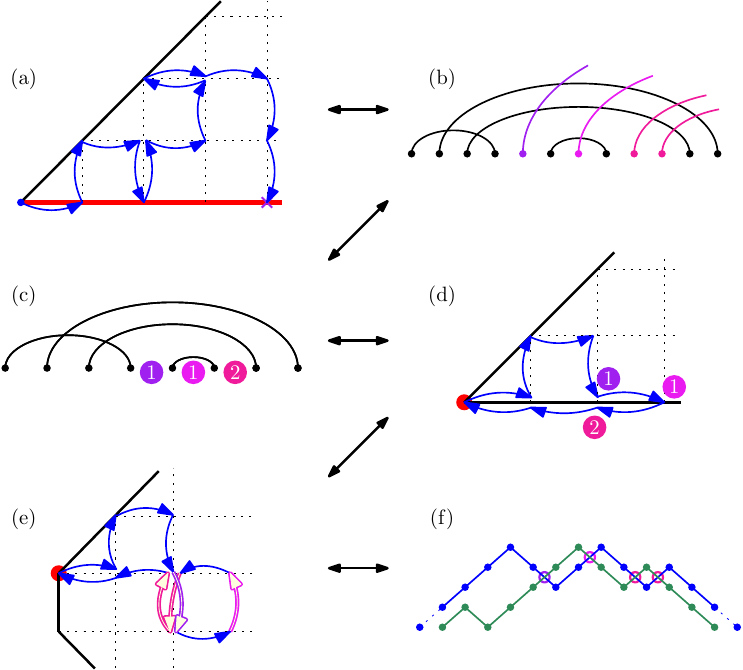}
 \caption{An example of the bijection of
   Theorem~\ref{theo:gouyou}. Successive objects are: (a) a
   simple axis-walk in the octant; (b) an open matching diagram with no $3$-crossing; (c) a
    matching diagram with no $3$-crossing and with integer weights at some intervals to record the positions of the former
open arcs; (d) a simple excursion
   in the octant where each visit to the $x$-axis carries
a nonnegative integer weight; (e) a simple excursion in the tilted quadrant; (f) a pair of
   Dyck paths whose lengths differ by $2$.}
 \label{fig:exOsc}
\end{figure}

As described in the introduction, our strategy relies on the bijection
 of Theorem~\ref{thm:BuCoFuMeMi}, which allows us to turn simple 
axis-walk into simple excursion with decorations consisting of
weights assigned to each visit to the $x$-axis.
\begin{lemma}
 \label{lem:oscRemOp}
 Simple axis-walks of length~$n$ in the octant ending at 
 $(m,0)$ are in bijection with simple excursions of length~$n-m$ in
 the octant, where each visit to the $x$-axis carries a non-negative
 integer weight, such that the sum of the weights is~$m$.
\end{lemma}

\begin{proof}
  Using Theorem~\ref{thm:BuCoFuMeMi}, such an axis-walk is mapped to
  an open matching diagram of size $n$ with $m$ open arcs and without
  $3$-crossing. We then remove the open arcs along with their nodes
   to obtain a (closed)
  matching diagram $\pi$ of size $n-m$, and we record their former
  positions as follows: for each interval of $\pi$ that contained at
  least one open arc, we assign to the interval a positive weight
  equal to the number of open arcs it formerly contained (see
  Figure~\ref{fig:exOsc}(b) to (c)). The sum of these weights is thus
  $m$. Only specific
  intervals can carry weights since adding arcs in some intervals
  might create a 3-crossing.

  By Theorem~\ref{thm:BuCoFuMeMi} (again), the diagram $\pi$ is mapped
  to an excursion in the octant. By Property~\ref{ppt:cddsy}, we know
  that the intervals of $\pi$ where insertion of open arcs is possible
  exactly correspond to the visits of the excursion to the $x$-axis.
  We then transfer the weights to the corresponding positions
(see Figure~\ref{fig:exOsc}(c) to (d)). 
\end{proof}

As a final step, we transform the weighted excursions in the octant
into pairs of Dyck paths. To do so, we define an intermediary class
of walks in the tilted quadrant \[\widetilde Q = \left\{(x,y): x\geq 0,
    \ |y|\leq x\right\},\]. This domain is constituted of two copies of the octant, a positive one $\{(x,y): x\geq y\geq 0\}$, and an upside-down negative one $\{(x,y): x\geq -y\geq 0\}$. 

\begin{lemma}
 \label{lem:oscEnd}
 For $n,m$ both even (resp. both odd), the set of simple decorated
 excursions of length~$n-m$ in the octant with a total weight $m$ on
 the visits to the $x$-axis is in bijection with the set of simple
 walks of length $n$ in the tilted quadrant $\widetilde Q$ from
 $(\frac 1 2,\frac 1 2)$ to $(\frac 1 2,\frac 1 2)$ (resp. to $(\frac
 1 2,-\frac 1 2)$) where exactly $m$ steps change the sign of $y$ in
 the walk.  This set is in bijection with $\cC_{\lfloor
   (n+1)/2\rfloor}\times\cC_{\lceil (n+1)/2\rceil}$, in such a way
 that if the two Dyck paths are drawn with respective starting points
 $((0,0),(-1,0))$, they cross exactly $m$ times.
\end{lemma}
\begin{proof}
  The idea is illustrated by
  Figure~\ref{fig:exOsc}\emph{(d)}--\emph{(f)}.
  
  First we translate the whole excursion by $(\frac 1 2,\frac 1 2)$.
  
  A weight on a position at height $\frac{1}{2}$ indicates the number of 
  \emph{switches} that are to be inserted at this position. The switches partition the path into a certain number of \emph{factors}. We apply the transformation $(x,y)\rightarrow (x,-y)$ to factors of even rank (where the first factor has rank $1$); and for each switch, we insert a vertical step crossing the $x$-axis in order to join the end of the preceding factor to the beginning of the next one.
  
Concerning the second bijection, we map any walk $(x_i,y_i)_{i\in
    \{0,\dots,n\}}$ of $\widetilde Q$ to the pair $P_1,P_2$ of
  paths \[\left((x_i+y_i)_{i\in \{0,\dots,n\}}, (x_i-y_i)_{i\in
      \{0,\dots,n\}}\right),\] i.e., the successive heights
of $P_1$ (resp. of $P_2$) are the successive values of $x_i+y_i$ 
(resp. of $x_i-y_i$). We easily see that the constraint of 
staying in $\widetilde Q$ is mapped to the constraint that both $P_1$ and $P_2$
remain nonnegative. In addition, for even length $2n$ 
 the endpoint conditions ensures that $P_1$ starts and
ends at $1$, while $P_2$ starts and ends at $0$; so $Q$ is a Dyck
path of length $2n$ and $P_1$ 
which can be identified with a Dyck path of length $2n+2$,
upon prepending an up-step and appending a down-step. For odd length
$2n+1$, the endpoint conditions ensure that $P$ starts at $1$ and 
ends at $0$, while $P_2$ starts at $0$ and ends at $1$; hence 
both $P_1$
and $P_2$ are identified with a Dyck path of length $2n+2$ (upon
prepending a down-step to $P_1$ and appending a down-step to $P_2$). 
\end{proof}
We
obtain the bijection for Theorem~\ref{theo:gouyou} by composing Lemma~\ref{lem:oscRemOp} with Lemma~\ref{lem:oscEnd}.

\section{Proof of Theorem~\ref{thm:MainBaxter}: Hesitating Walks}
\label{sbs:mainHes}
We next give the details of the bijection between hesitating
axis-walks in the octant, and hesitating excursions to prove
Theorem~\ref{thm:MainBaxter}. The initial part is similar to before,
and we provide two variants for the second.

\subsection{Transformation into decorated hesitating excursions in the octant}
The general strategy is the same as for simple walks: turn an
axis-walk in the octant into an open partition diagram, remove the open arcs while marking their locations, and transform the decorated diagram back to a decorated excursion. In contrast, the
second part is different from the simple walk case, since the marking does not easily induce an excursion in a larger domain. That is why we need an
additional step of decoration transfer. 

\begin{lemma}
\label{prop:openHes2}
Hesitating walks of length $2n$ staying in the octant and ending at
$(m,0)$ are in bijection with hesitating excursions of length $2n$
staying in the octant in which $m$ W-steps on the $x$-axis have been
marked.
\end{lemma}
\begin{figure}\center
 \includegraphics[width=\textwidth]{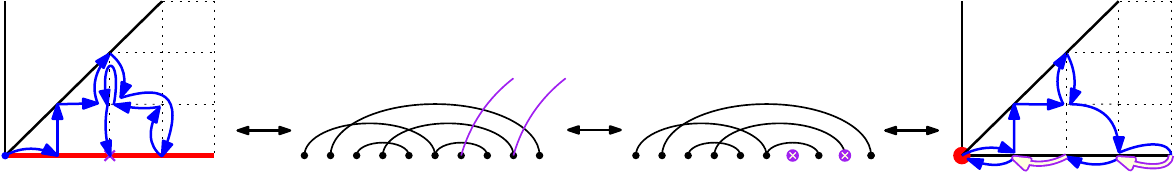}
 \caption{An example of the first part of the bijection.  From left to
   right: a hesitating axis-walk in the octant; an open partition
   diagram; a decorated partition diagram; a decorated hesitating
   excursion in the octant}
 \label{fig:remOpArcHes}
\end{figure}

\begin{proof}
  Using Property~\ref{ppt:cddsy}, it is easy to check that for $\pi$
  a partition-diagram of size $n$ and $\omega$ the corresponding
  hesitating excursion of length $2n$ in the octant, the (closing or fixed)  points
  of $\pi$ where an open arc can be added exactly
  correspond to the $W$-steps of $\omega$ on the $x$-axis. If we mark
  $m$ such steps we obtain an open partition diagram of size $n$ with
  $m$ open arcs and no enhanced $3$-crossing, which itself corresponds
  (by Theorem~\ref{thm:BuCoFuMeMi}) to an hesitating walk of length
  $2n$ in the octant that ends at $(m,0)$.
\end{proof}

It turns out that hesitating excursions in the quadrant (the second family of objects appearing in Theorem~\ref{thm:MainBaxter}) are also in bijection with \textit{some} decorated hesitating excursions in the octant, as stated by the following lemma.

\begin{lemma}
\label{lem:mirror}
Hesitating (resp. simple) excursions of length $2n$ in the first
quadrant are in bijection with hesitating (resp. simple) excursions of
length~$2n$ in the first octant with marked steps leaving the diagonal
$y=x$.
\end{lemma}
 As proved in~\cite{BuCoFuMeMi15}, this lemma is a consequence of the reflection principle with
respect to the diagonal.

In terms of hesitating excursions of the quadrant, the number of marked steps corresponds to
a parameter called the \emph{switch-multiplicity} of the walk, which
is roughly speaking the number of times that the walk crosses the
diagonal, or similarly to the number of times the walk goes over the
diagonal.

\subsection{Moving the marks around}
In view of Lemmas~\ref{prop:openHes2} and~\ref{lem:mirror}, Theorem~\ref{thm:MainBaxter} holds if there is an equidistribution for the hesitating excursions of the octant
between the parameter counting the steps leaving the diagonal, and the
parameter counting~$W$-steps on the $x$-axis. This is true, and
furthermore, they are symmetrically distributed.

\begin{proposition}
\label{prop:involution}
  There is an explicit involution over the set of hesitating
  excursions of length $2n$ in the octant that exchanges the number of
  W-steps on the $x$-axis and the number of steps leaving the
  diagonal.
\end{proposition}

\begin{proof}
  The proof passes through four main intermediaries. We first map a
  hesitating walk to a simple walk, tracking enough information to
  recover the hesitating walk. We use a classic mapping of steps to
  convert a simple excursion to a pair of Dyck paths. Then, we apply an
  involution on pairs of Dyck paths which swaps a key parameter. The
  final deduction comes from tracing parameters through these
  bijections, and back. This gives the stated result.

  We now give the details of the individual steps.
  
\vspace{1em}
  
\renewcommand{\descriptionlabel}[1]{\hspace{\labelsep}\textbf{#1}.}
\begin{description}[leftmargin=0em, itemsep=2em]
\item[From hesitating walks to simple walks] 
We call a \textit{sailing point} a positive step, $E$ or $N$, followed by a
  negative step, $W$ or $S$.
We transform every
  hesitating walk into a simple walk in which some sailing
    points are marked. To do so, we gather the steps of
  the hesitating walk in pairs, discarding the zero-steps, and marking
  every sailing point induced by the gathering of two non-zero steps.
Thus, every hesitating excursion in the octant is identified
with a simple excursion in the octant where some sailing
points are marked.  
\item[From simple walks to pairs of Dyck paths] Simple excursions in
  the octant are mapped to non-crossing pairs of Dyck paths by to the
  transformation $(x_i,y_i)\rightarrow((x_i+y_i),(x_i-y_i))$.  The sailing points of the
  excursion become the \textit{upper peaks}, i.e. the peaks of the upper path; the W-steps on the
  $x$-axis become the \emph{upper contacts}, i.e. down steps occurring
  at the same time and the same height for both paths; and the steps
  leaving the diagonal become the \emph{lower contacts}, i.e. up steps
  leaving the $x$-axis in the lower path. The proof is now reduced to
  the following claim.
\item[Involution for non-crossing pairs of Dyck paths] 
  \begin{claim}
  \label{clm:inv}
  There is an explicit involution on pairs of non-crossing Dyck paths of length
  $2n$, which preserves the number of upper peaks and maps the number of
  upper contacts to the number of lower contacts.
  \end{claim}
  Such an involution is given by an article of Elizalde and
  Rubey~\cite{ElRu16}. More specifically, their involution operates on
  non-crossing pairs of Dyck paths, preserves the upper path (hence
  the upper peaks), and exchanges the number of upper and lower
  contacts. 
  
  In the next section, we also give an alternative proof of the previous claim, using Schnyder woods.
\end{description}
\end{proof}

\subsection{Bonus: Vacillating Walks}

In the literature (see \cite{ChDeDuStYa07,BuMeMi15} for example), there are commonly three families of walks related to tableaux: the simple and hesitating walks, which we have already dealt with, and vacillating  walks. 
A (2-dimensional) \emph{vacillating} walk is an even sequence of steps
such that every step of odd index is 
 $W$, $S$ or $\mathbf{0}$ (a negative or zero step), and every step of
even index is  
$N$, $E$ or $\mathbf{0}$ (a positive or zero step). 
Unlike hesitating walks, vacillating walks can have successive zero steps.

It turns out that an analogue of Theorems~\ref{thm:MainSimple} and~\ref{thm:MainBaxter} for vacillating walks can be seen as a corollary of Theorem \ref{thm:MainBaxter}.

\begin{corollary}
For every integer $n \geq 1$, the number of vacillating axis-walks of half-length~$n$ in the first octant is twice the number of vacillating excursions of half-length~$n$
in the first quadrant. The latter number is equal to
 $\sum_{k=0}^{n-1}\binom{n-1}{k}B_{k + 1}$, where $B_k$ is the $k$th Baxter number. 
\end{corollary}

\begin{proof} First, we claim that the vacillating axis-walks of half-length $n$ in the quadrant (resp. in the octant) are in bijection with the triples $(hw,P,laststep)$ where:
\begin{itemize}
\item $hw$ is a hesitating axis-walk of half-length $k \leq n - 1$ in the quadrant (resp. in the octant);
\item $P$ is a multiset of  size $n-1 - k$ and 
with elements in $\{0,\dots,k\}$;
\item $laststep$ is either $E$ or $\mathbf 0$.  
\end{itemize}
Indeed, removing the first step and the last step from a vacillating walk $w$ of half-length $n$ induces a hesitating walk of half-length $n-1$ where we allow some consecutive steps $s_{2i-1}$ and $s_{2i}$ to be
zero. Thus, to obtain a valid hesitating walk $hw$, we have to remove the consecutive double-zero steps, and store them in a multiset $P$, which describes the positions where we have to insert back the double-zero steps to recover the original vacillating walk. 
To recover the original vacillating walk we also have to put back the first step, which has to be $\mathbf 0$, because a walk starting at the origin cannot begin by a negative step, and the last step, which can be $E$ or $\mathbf 0$, but not $N$ because the walk ends on the axis.

If we restrict the axis-walk to be an excursion, then under the previous bijection, $hw$ is also an excursion, and $laststep$ is necessarily $\mathbf 0$ (hence the factor $2$).

We now use the bijection of Theorem~\ref{thm:MainBaxter}, which matches the hesitating  axis-walks in the first octant and the hesitating excursions in the first quadrant. Translated in terms of vacillating walks via the above bijection, it means that the vacillating axis-walks of the first octant are in bijection with the pairs formed by a vacillating excursion in the first quadrant, and a step in $\{E,\mathbf 0\}$. 

As for the counting formula for vacillating excursions of half-length $n$ in the first quadrant, it directly follows from the fact that hesitating excursions of half-length $k$ in the first quadrant are counted by $B_{k+1}$. 
\end{proof}

\section{Alternative Proofs Using Schnyder Woods}
\label{sec:SchWd}

In this section, we propose some alternative proofs of Theorem~\ref{thm:eliz} and Claim~\ref{clm:inv} by using as an intermediate step a famous combinatorial family: the \textit{Schnyder woods}.

\subsection{Schnyder woods and pairs of non-crossing Dyck paths}

A \textit{rooted triangulation} is the embedding of a simple planar
graph in the plane such that all faces are triangles, with a
distinguished \textit{external face} and a distinguished \textit{root
  vertex} on the external face. We are interested in a particular kind
of edge-coloring and orientation of triangulations, known as \textit{Schnyder
  woods}. One of them is depicted by Figure~\ref{fig:exSWOsc}(d). Schnyder woods were introduced by Schnyder in~\cite{Schn89}
and \cite{Schn90} for triangulations and were later extended to
different families of maps. There are several applications of these colorings, ranging
from graph drawing to map encoding.

We describe a Schnyder wood using some conventions of the
triangulation. The \textit{size} of a triangulation is the number of
vertices minus 3 (we do not count the vertices of the external
face). The root vertex is labeled~$v_0$, and the two other vertices
in counterclockwise order around the external face are labeled~$v_1$
and $v_2$. The vertices that are not incident to the external face and
edges that do not bound the external face are called
\textit{internal}.

A Schnyder wood of a simple triangulation is an edge-coloring into
$3$ colors, along with an orientation of all internal edges, such that
the following properties are satisfied:
\begin{itemize}
\item for each $i \in \{0,1,2\}$, the set of $i$-colored edges forms a
  spanning tree $T_i$, rooted and oriented towards $v_i$;
\item if, for $i \in \{0,1,2\}$, an $i$-tail (resp. $i$-head) with
  respect to a vertex denotes an edge colored by $i$ oriented away
  from (resp. toward) this vertex, then in clockwise order around any
  internal vertex, there are: one $0$-tail, some $1$-heads, one
  $2$-tail, some $0$-heads, one $1$-tail, some $2$-heads.
 \end{itemize}

Schnyder woods are in bijective correspondence with several
combinatorial families. A pair $P,Q$ of Dyck paths of length $2n$
(both starting at $(0,0)$ and ending at $(2n,0)$) is called
\emph{non-crossing} if for each $0\leq i\leq 2n$, the height of $Q$ after $i$ steps is at most the height of $P$ after $i$ steps.

\begin{theorem}[Bonichon~\cite{Boni05}]
 \label{thm:BeBo}
 Non-crossing pairs of Dyck paths of length~$2n$ are in bijection with Schnyder woods of size $n$.
\end{theorem}
Theorem~\ref{thm:BeBo} was first proved by Bonichon in \cite{Boni05}.
Bernardi and Bonichon improved the description of this bijection a few
years later~\cite{BeBo09}. It is the latter form that we use
here. We give a quick description of the map~$\Psi$ from Schnyder
woods to pairs of non-crossing Dyck paths -- see (c) and (d) from 
Figure \ref{fig:exSWOsc} for an example of the bijection. First, remember the
classical bijection~$\Omega$ 
between plane trees and Dyck paths: take
a plane tree, turn around it clockwise, starting and ending at the
root; the first time an edge is visited, write an up-step, the second
time, write a down step.  From a Schnyder wood, we now generate a pair
$(P,Q)$ of non-crossing Dyck paths. The bottom path
$Q=UD^{\alpha_1}\ldots UD^{\alpha_n}$, where $U$ stands for up-step and
$D$ stands for down-step, is the path representing the
tree $T_0$ of color $0$: $Q=\Omega(T_0)$. The tour around $T_0$
induces an order on the internal vertices, that we subsequently call
$u_1,...,u_n$ (the first vertex visited by the tour is $v_0$). Let
$\beta_i$ be the number of $1$-heads incident to $u_i$ and let
$\beta_{n+1}$ be the number of 1-heads incident to $v_1$. Note that
$\beta_1$ has to be $0$. The upper path is now defined as
$P=UD^{\beta_2}...UD^{\beta_{n+1}}$.  We refer the reader to Bernardi and Bonichon's paper \cite{BeBo09} for a
proof that $P$ is positive and does not cross $Q$, a description of
the reverse map $\Phi$, and a proof that $\Psi$ and $\Phi$ are
reciprocal bijections.

For our own purpose, we need to review how certain parameters are 
transformed by $\Psi$, as shown in Table \ref{tbl:paramBB}
where the first column corresponds to Schnyder woods, and the second
column corresponds to pairs of Dyck paths. The bijections $\Psi$ and
$\Phi$ map each parameter to its counterpart on the same row. In the
proof of Proposition~\ref{prop:involution}, we defined the \emph{upper
contacts} as common down-steps. Similarly, \emph{reversed upper
contacts} are  common up-steps. \emph{Upper} (resp. \emph{lower}) \emph{peaks} are peaks of the upper (resp. lower) path.

\begin{lemma}
 \label{lem:refBeBo}
 The bijection $\Phi$ satisfies the correspondence of parameters given by Table \ref{tbl:paramBB}.
\end{lemma}

\begin{table}

 \begin{tabular}{l l l}
\toprule
  & Schnyder wood & pair of non-crossing Dyck paths\\
  \toprule  1&size & half-length \\  \midrule
    2&number of leaves of $T_0$ & number of lower peaks \\ \midrule
    3&number of internal nodes of $T_1$ & number of upper peaks \\  \midrule
    4&degree of $v_0$ in $T_0$ & number of lower contacts \\  \midrule
    5&degree of $v_1$ in $T_1$ & length of the last upper descent \\ \midrule
    6&degree of $v_2$ in $T_2$ & number of reversed upper contacts    
     \\ \bottomrule
 \end{tabular}
\smallskip

 \caption{\label{tbl:paramBB} The correspondence of parameters through the  bijections $\Psi$ and $\Phi$.  }
\end{table}

\begin{proof}
  This proof uses references and notation from Bernardi and Bonichon's paper ~\cite{BeBo09}. Row~1 is
  trivial, and Rows~2 and 4 are well-known properties on the bijection
  between Dyck paths and trees. Row~5 is a direct consequence of the
  construction of the bijection, and Row~3 is a direct consequence of
  the fact that an upper peak is a descent $i$ of positive length
  $\beta_i>0$, which corresponds to an internal node of $T_1$.
 
  Row~6 is less straight-forward. It corresponds to the tight case
  in~\cite{BeBo09} when proving that the pair of walks is
  non-crossing; we sketch the main arguments.  Let $u_i$ be an
  internal vertex of $T$ (for $1\leq i\leq n$), and let $h_i$ be the
  tail of color $2$ at $u_i$.  Then one can check that $u_i$ is a
  neighbor of $v_2$ in $T_2$ if and only if $h_i$ is not \emph{below}
  a 1-arc (an arc of color $1$), i.e., there is no 1-arc $e$ such that $h_i$ is inside the
  (unique) cycle formed by $e$ and $T_0$.
  In such a situation, $h_i$ comes after the tail
  of $e$ and before the head of $e$ during a clockwise tour around
  $T_0$.  The number of 1-tails before $h_i$ in such a tour is
  $\sum_{j<i}\alpha_j$, while the number of 1-heads
  before $h_i$ is $\sum_{j\leq i}\beta_j$.
  Consequently, given that a 1-tail always comes in the tour before its corresponding 1-head, 
  $u_i$ is a neighbor of $v_2$ in $T_2$ if and only if
  the previous numbers coincide, i.e.,
  $\sum_{j<i}\alpha_j=\sum_{j\leq i}\beta_j$.  Finally, we remark that
  $-i+\sum_{j<i}\alpha_j$ (resp. $-i+\sum_{j\leq i}\beta_j$) gives the
  height of the $i$th up-step of the lower (resp. upper) Dyck path;
  hence the two numbers match if and only if the $i$th up-steps of the lower and
  upper Dyck paths form a reversed upper contact.
\end{proof}

\subsection{Consequences}

From Lemma~\ref{lem:refBeBo}, we can prove several non-trivial
properties on non-crossing pairs of Dyck paths, and subsequently on
excursions in the octant. Schnyder woods have
more evident symmetries than pairs of Dyck paths, and the very
expression of these symmetries gives involutions that are not easily
phrased in terms of pairs of Dyck paths.

We give an alternative proof of Theorem~\ref{thm:eliz} (see
Figure~\ref{fig:exSWOsc} for an outline of the bijection).

\begin{figure}
 \includegraphics[width=.9\textwidth]{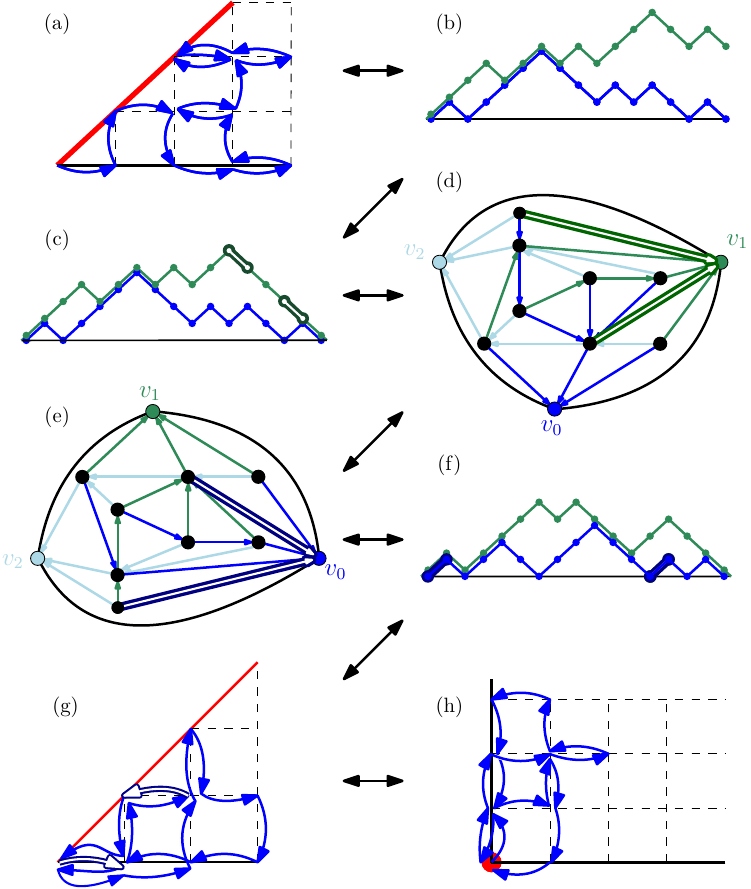}
 \caption{Illustration of the bijective proof of
   Theorem~\ref{thm:eliz} using Schnyder woods. Successive objects are: (a)
a simple walk ending on the diagonal and staying in the octant; 
a non-crossing pair formed by a Dyck path and a non-negative walk starting at $(0,0)$;
(c)
 a non-crossing pair of Dyck paths where some steps on
   the last descent of the upper path are marked;
    (d) and (e) two Schnyder woods (the roles of $T_0$ and $T_1$ have been swapped thanks to a vertical reflection);
   (f) a non-crossing pair of Dyck paths where some up-steps of the lower
   path that leave the $x$-axis are marked;
   (g) a simple excursion in the quadrant where some steps leaving the diagonal are marked;
   (h) a simple  excursion staying in the quadrant.}
 \label{fig:exSWOsc}
\end{figure}

\begin{proof}[Proof of Theorem~\ref{thm:eliz}]
We are going to describe the bijection from excursions to walks ending at the diagonal.

First, we use Lemma~\ref{lem:mirror} to map an excursion in the quadrant to an excursion in the octant with $k$ marked steps leaving the diagonal. This excursion is then mapped to a pair of non-crossing Dyck paths with $k$ marked lower contacts. 
 
We apply $\Phi$, move the root from $v_0$ to $v_1$ and change the orientation (meaning that ``clockwise'' becomes ``counterclockwise'').  Moving the root and changing the orientation amounts to exchanging the roles of $T_0$ and $T_1$. Hence, according to Lemma~\ref{lem:refBeBo} (rows 4 and 5), when we then apply $\Psi$ to get back to a pair of non-crossing Dyck paths, we get $k$ marked steps on the descent of the upper path.  
 
We reverse these $k$ steps, and the upper path now ends at height $2k$. We map the pair of paths back to a walk in octant that ends at coordinates $(k,k)$.
\end{proof}

A similar strategy can be applied to prove the next theorem, which is a stronger version of Claim~\ref{clm:inv} (that was used to complete the proof of Theorem \ref{thm:MainBaxter}), but still weaker than the result from \cite{ElRu16}, which preserves the upper path.

\begin{theorem}
 \label{thm:SchHesSym}
 There is an explicit involution on pairs of non-crossing Dyck paths of length
  $2n$, which preserves the number of upper peaks and exchanges the number of
  upper contacts and the number of lower contacts.
\end{theorem}

\begin{proof}
 We take a pair $(P,Q)$ of non-crossing Dyck paths, reverse them to
 transform upper contacts into reversed upper contacts, while keeping
 the same number of peaks and lower contacts, and apply $\Phi$ to get a
 Schnyder wood. 
 
 Then we move the root of the map from $v_0$ to $v_2$, and
 flip the orientation of the plane. This has the effect of exchanging the roles 
 played by $T_0$ and $T_2$.
 
 Finally we apply $\Psi$ to get back to a
 pair of non-crossing Dyck paths, and reverse again the two paths. 
 Lemma~\ref{lem:refBeBo} (rows 3, 4 and 6) is enough to conclude the proof.
\end{proof}

\subsection{An extension of the Narayana symmetry}
The Narayana number $\mathrm{N}(n,p)$ is defined as the number of Dyck
paths of length $2n$ with $p$ peaks. These numbers refine the
Catalan numbers $\mathrm{Cat}_n$, in the sense that
$\sum^{n}_{p=1}\mathrm{N}(n,p)=\mathrm{Cat}_n$. The following
statement is well-known.

\begin{property}
 \label{prop:Nara}
 The Narayana numbers satisfy the following symmetry property
 \[\mathrm{N}(n,p)=\mathrm{N}(n,n-p+1).\]
\end{property}

\begin{figure}
 \includegraphics[width=\textwidth]{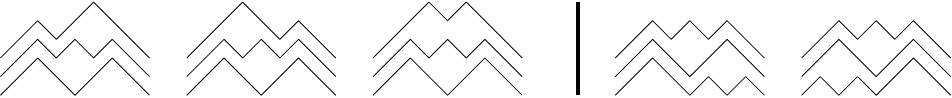}
 \caption{A minimal counterexample to the conjectural identity $\mathrm{N}(n,p_1,...p_k)=\mathrm{N}(n,n-p_k+1,...n-p_1+1)$
(which holds for $k\in\{1,2\}$).  
We have $3=\mathrm{N}(4,2,3,2)\neq\mathrm{N}(4,3,2,3)=2$.}
 \label{fig:ctex}
\end{figure}

The symmetry can be obtained from a classical bijection between Dyck
paths of length $2n$ and rooted binary trees with $n+1$ leaves: the
number of peaks of the Dyck path is mapped to the number of left
leaves, and the symmetry follows by applying a reflexion to the tree.

The Narayana numbers can be extended to any $k$-tuple of non-crossing
Dyck paths in the following way: $\mathrm{N}(n,p_1,...,p_k)$ is the
number of non-crossing $k$-tuples $D_1,\ldots,D_k$ of Dyck paths
(ordered from bottom to top) such that $D_i$ has $p_i$ peaks.

\begin{theorem}
 \label{thm:Nara2}
For $k=2$, the extended Narayana numbers satisfy the symmetry property
 \[
\mathrm{N}(n,p,q)=\mathrm{N}(n,n-q+1,n-p+1).
\]
\end{theorem}

\begin{proof}
 The method we use is similar both to the case of classical Narayana numbers and to the previous subsection.   
Starting from a non-intersecting pair of Dyck paths, we apply $\Phi$,  move the root from $v_0$ to $v_1$, change the orientation (similarly as before, this amounts to exchanging the roles of $T_0$ and $T_1$), and apply 
$\Psi$ back to a pair of paths. This yields an involution on
non-crossing pairs of Dyck paths that has the desired peak-parameter
correspondence, according to Lemma~\ref{lem:refBeBo} (rows 2 and 3). 
\end{proof}

However, a similar symmetry does not seem to hold for higher values of $k$. For example, one could expect that $\mathrm{N}(n,p_1,\dots,p_k)=\mathrm{N}(n,n-p_k+1,\dots,n-p_1+1)$, but we present a minimal counterexample to that in Figure~\ref{fig:ctex}.

\section{A new bijection for Young tableaux of even-bounded height}
\label{sec:young}
As we have seen in Section~\ref{sbs:mainOsc}, the main step in the
proof of Theorem~\ref{thm:MainSimple} is an explicit bijection between
simple axis-walks of length $n$ staying in the octant
$\{x\geq y\geq 0\}$, and simple walks of length $n$ from
$(\frac 1 2,\frac 1 2)$ to $(\frac 1 2,\frac {(-1)^n} 2)$ staying in
the tilted quadrant.  As it turns out, this bijection can be easily
generalized to any dimension, and infers new connections with standard Young tableaux with even-bounded height.

\subsection{Walks in higher dimensional Weyl Chambers}
For $k\geq 1$, we define the $k$-dimensional Weyl
chamber\footnote{For convenience we define the chambers using
  non-strict inequalities, our bijective statements can equiva\-lently   be given under strict inequalities, upon applying the
  coordinate shift $\widetilde{x}_i=x_i+k+1-i$.}  of type C as
\[W_C(k):=\left\{ (x_1,x_2,\dots,x_k)\ \ |\ \ x_1 \geq x_2 \geq \dots
  \geq x_{k} \geq 0\right\},\]
and the $k$-dimensional Weyl chamber of type D as
\[W_D(k):=\left\{ (x_1,x_2,\dots,x_k)\ \ |\ \ x_1 \geq x_2 \geq \dots
    \geq x_{k-1} \geq |x_k|\right\}.\]
  In this context, an \emph{axis-walk} is any walk starting
at the origin and ending on the $x_1$-axis. With these definitions,
the generalization of Theorem~\ref{thm:MainSimple} reads as follows.
 \begin{theorem}\label{theo:Wey}
   For~$k\geq 1$ and~$n\geq 0$, there is an explicit bijection between
   simple axis-walks of length $n$ staying in $W_C(k)$ and simple
  excursions of length $n$ staying in $W_D(k)$, starting from
   $(\frac 1 2,\dots,\frac 1 2,\frac 1 2)$, and ending at
   $(\frac 1 2,\dots,\frac 1 2,\frac {(-1)^n} 2)$. The ending
   $x_1$-coordinate of a walk from $W_C(k)$ corresponds to the number
   of steps that change the sign of $x_k$ in its bijective image.
\end{theorem}
Note that the case $k=1$ is precisely our
introductory example, and the case $k=2$ is the first
part of Lemma~\ref{lem:oscEnd}.    
The arguments to show Theorem~\ref{theo:Wey} 
are very similar to those in the proofs of
Lemma~\ref{lem:oscRemOp} and the first part of Lemma~\ref{lem:oscEnd}. They
use the general formulation of Theorem~\ref{thm:BuCoFuMeMi}
(specifically, the bijection between open matching diagrams without $(k+1)$-crossing and
simple axis-walks in $W_C(k)$), and the property that
the intervals where an open arc can be added (without creating a
$(k+1)$-crossing) correspond to the visits of the walk to the
hyperplane defined by $x_k=0$.


Moreover, it has been
recently shown~\cite{BuCoFuMeMi15, Krat16} that
the set of standard
Young tableaux of size $n$ with height at most $2k$ is in bijection with simple axis-walks of
length $n$ in $W_C(k)$, with the ending $x_1$-coordinate mapped to the
number of columns of odd length.  Composing this bijection with
Theorem~\ref{theo:Wey} infers the following result.

\begin{corollary}\label{theo:young}
  For $n,k\geq 1$, there is an explicit bijection between
  the  standard
Young tableaux of size $n$ with height at most $2k$, and the simple walks of length $n$ staying
  in $W_D(k)$, 
starting from $(\frac 1 2,\dots,\frac 1 2,\frac 1 2)$, and
  ending at $(\frac 1 2,\dots,\frac 1 2,\frac {(-1)^n} 2)$. The number
  of odd columns corresponds to the number of steps that change the sign
  of $x_k$.
\end{corollary}

\subsection{Recovering Gessel's formula}
\label{sec:Gessel}
Thanks to the lattice path enumeration techniques of Grabiner and
Magyar~\cite{GrMa93}, the previous corollary has an interesting
consequence: a combinatorial interpretation of the determinant
expression of Gessel~\cite{Gess90} for the generating function of
standard Young tableaux of even-bounded height.

\begin{proposition}[Gessel~\cite{Gess90}]\label{prop:gessel}
Let $Y_d[n]$ be the number of Young tableaux of size $n$ with at most $d$ rows, and $Y_d(x)=\sum_{n\geq 0}\frac1{n!}Y_d[n]x^n$ the associated generating function. Then for each $k\geq 1$,
\begin{equation}\label{eq:gessel}
Y_{2k}(x)=\mathrm{det}\big(I_{i-j}(2x)+I_{i+j-1}(2x)\big)_{1\leq i,j\leq k},
\end{equation}
where (for $m\in\mathbb{Z}$) 
$I_m(2x)=\sum_{i\geq 0}\frac1{(m+i)!i!}x^{m+2i}$. 
\end{proposition}

  
Let us now explain how we can recover this result from 
Corollary~\ref{theo:young}. 
First, it proves to be convenient to take here the Weyl chamber of type D under
the form \[\wtW_D(k):=\{|x_1|<x_2<\cdots<x_k\}.\] 
For each point $\lambda=(\lambda_1,\ldots,\lambda_k)$ (in $\mathbb{R}^k$) we denote by $\lambda'$ the point $(-\lambda_1,\lambda_2,\ldots,\lambda_k)$. 
Let $\rho=(1/2,3/2,\ldots,k-1/2)$.
Then Corollary~\ref{theo:young}
states that for $n$ even (resp. odd), 
$Y_{2k}[n]$ is the number of walks of length $n$ in 
$\wtW_D(k)$ from $\rho$ to $\rho$ (resp. to $\rho'$). 

For every points $\lambda=(\lambda_1,\ldots,\lambda_k)$ and
$\mu=(\mu_1,\ldots,\mu_k)$ both in $\wtW_D(k)$, let $N_{\lambda,\mu}[n]$ be the number
of simple walks of length $n$ from $\lambda$ to $\mu$ staying in
$\wtW_D(k)$, and let $N_{\lambda,\mu}(x):=\sum_{n\geq
  0}\frac1{n!}N_{\lambda,\mu}[n]x^n$ be the associated generating
function.

\begin{lemma}[Grabiner and
Magyar~\cite{GrMa93}]\label{lem:gen_funct}
For every points $\lambda=(\lambda_1,\ldots,\lambda_k)$
and $\mu=(\mu_1,\ldots,\mu_k)$ in $\wtW_D(k)$ such that $\lambda_i$ and $\mu_i$ belong to $1/2 + \mathbb Z$ for  $i \in \{1,\dots,k\}$, we have 
$$
N_{\lambda,\mu}(x)+N_{\lambda,\mu'}(x)=\mathrm{det}\big(I_{\lambda_i-\mu_j}(2x)+I_{\lambda_i+\mu_j}(2x)\big)_{1\leq i,j\leq k},
$$
where, for $m\in\mathbb{Z}$, $I_m(2x)$ is defined as 
$I_m(2x)=\sum_{i\geq 0}\frac1{(m+i)!i!}x^{m+2i}$. 
\end{lemma}
The proof techniques in~\cite{GrMa93} 
rely on a general reflexion principle
(see~\cite[Theorem 1]{GeZe92}) for Weyl chamber walks,
which results in determinant expressions for the 
relevant generating functions (the determinant
is naturally expressed in terms of $I_m(2x)$, which 
is the exponential generating function of simple 1d walks
that start at $0$ and end at $m$).   


By Corollary~\ref{theo:young} we have
\[Y_{2k}(x)=N_{\rho,\rho}(x)+N_{\rho,\rho'}(x),\]
where $N_{\rho,\rho}(x)$ gathers the coefficients of even power
and $N_{\rho,\rho'}(x)$ gathers the coefficients of odd power.   Hence, applying
Lemma~\ref{lem:gen_funct} to $\lambda=\mu=\rho$, we recover
Proposition~\ref{prop:gessel}.

\subsection{Related formulas}
For odd $d$, the expression for
the generating function $Y_{d}(x)$ found in Gessel's paper~\cite{Gess90} is 
\[
Y_{2k+1}(x)=e^x\ \!\mathrm{det}\big(I_{i-j}(2x)-I_{i+j}(2x)\big)_{1\leq i,j\leq k}. 
\] 
In that case the combinatorial derivation is easier.  Indeed, a
Young tableau with at most $2k+1$ rows is identified (via the
Robinson-Schensted correspondence) to an involutive permutation
without $(2k+2)$-decreasing subsequence, which itself maps to a
partial matching diagram without $(k+1)$-nesting (\textit{partial} here means
that there can be isolated points along the line). This implies that
\[
\sum_{n\geq 0}\frac1{n!}Y_{2k+1}(n)x^n=e^x\ \! M_k(x),
\]
where $M_k(x)$ is the exponential generating function for matching
diagrams without $(k+1)$-nesting. By the Chen \emph{et al.}
bijection~\cite{ChDeDuStYa07}, the series $M_k(x)$ is the exponential generating
function of simple walks in the Weyl chamber of type $C$ starting and
ending at the origin; and it is shown by Grabiner and
Magyar~\cite[Section 6.2]{GrMa93} that
this generating function is
$\mathrm{det}\big(I_{i-j}(2x)-I_{i+j}(2x)\big)_{1\leq i,j\leq k}$.

Similarly, a determinant formula is known for the enumeration of
\emph{pairs} of Young tableaux of bounded height.  More precisely, let
$u_d[n]$ be the number of pairs of Young tableaux of the same shape
with at most $d$ rows (also by the Robinson-Schensted correspondence, the number of permutations in
$\mathfrak{S}_n$ with no $(d+1)$-increasing subsequence). Then for every $d \geq 1$, we
have as shown by Gessel ~\cite{Gess90} (note that this time the expression is uniform in
$d$, with no dependence on the parity)
\[
\sum_{n\geq 0}\frac{u_d[n]}{n!^2}x^{2n}=\mathrm{det}\big(I_{i-j}(2x)\big)_{1\leq i,j\leq k}.
\]
A combinatorial proof of this expression has been given in~\cite{gessel1998lattice}
via simple walks ending at so-called Toeplitz points, see 
also~\cite{Xin10} 
for a combinatorial derivation based on arc diagrams.

Regarding asymptotic enumeration, it should in principle be possible
to use recent results by Denisov and Wachtel for the asymptotic
enumeration of walks in cones~\cite[Theo.~6]{denisov2015random} in order to recover
from Theorem~\ref{theo:young} the expression found by
Regev~\cite{regev1981asymptotic} for the asymptotic number of Young
tableaux of size $n$ with at most $2k$ rows, which is, for each fixed
$k\geq 1$,
\[Y_{2k}[n]\sim_{n\to\infty} (2/\pi)^{k/2}(2k)^n(k/n)^{k(k-1/2)}\prod_{i=0}^{k-1}(2i)!.\] 
(The relevant constants for the Weyl chamber of type D can be computed
from~\cite{konig2009random} and from Selberg
integrals~\cite[Eq. 1.20]{forrester2008importance}.)

\section{Acknowledgements} The authors thank Guillaume Chapuy
and Alejandro Morales for interesting discussions.  MM, JC and ML were
partially supported by NSERC Discovery Grant 31-611453, and EF
was partially supported by PIMS.

\bibliographystyle{plain}
\bibliography{main.bib}
\end{document}